\documentclass[11pt]{article}

\usepackage{amsmath,amssymb,amsthm} 
\usepackage{enumitem} 
\usepackage{cite} 
\usepackage[margin=1in]{geometry} 
\usepackage{hyperref}  
\hypersetup{colorlinks,linkcolor=[rgb]{0.0,0.0,0.4},citecolor=[rgb]{0.5,0.0,0.0}}

\newcommand{\free}[1]{\left\langle#1\right\rangle} 
\newcommand{\norm}{\operatorname{Norm}}
\newcommand{\disc}{\operatorname{disc}}
\newcommand{\tr}{\operatorname{tr}}
\newcommand{\Span}{\operatorname{Span}}
\newcommand{\End}{\operatorname{End}}

\newcommand{\BB}[1]{\mathbb{#1}} 
\newcommand{\script}[1]{\mathcal{#1}} 
\newcommand{\mfrak}[1]{\mathfrak{#1}} 
\newcommand{\CC}{\BB{C}}
\newcommand{\QQ}{\BB{Q}}
\newcommand{\ZZ}{\BB{Z}}
\newcommand{\FF}{\BB{F}}
\newcommand{\sO}{\script{O}}

\newcommand{\px}{P}
\newcommand{\po}{P_0}
\newcommand{\fl}{f_1}
\newcommand{\fll}{f_2}

\theoremstyle{plain}
\newtheorem{thm}{Theorem}
\newtheorem{lem}[thm]{Lemma}
\newtheorem{cor}[thm]{Corollary}

\theoremstyle{remark}
\newtheorem{rem}[thm]{Remark}

\theoremstyle{definition}
\newtheorem{defn}[thm]{Definition}
\newtheorem{ex}[thm]{Example}
\newtheorem{heuristic}[thm]{Heuristic}

\title{Super-Isolated Elliptic Curves and Abelian Surfaces in Cryptography}
\author{
	Travis Scholl
	\\ Department of Mathematics
	\\ University of Washington
	\\ \href{mailto:tscholl2@uw.edu}{tscholl2@uw.edu}
}

\begin{document}

\maketitle

\abstract{We call a simple abelian variety over $\FF_p$ \emph{super-isolated} if its ($\FF_p$-rational) isogeny class contains no other varieties. The motivation for considering these varieties comes from concerns about isogeny based attacks on the discrete log problem. We heuristically estimate that the number of super-isolated elliptic curves over $\FF_p$ with prime order and $p \leq N$, is roughly $\tilde{\Theta}(\sqrt{N})$. In contrast, we prove that there are only 2 super-isolated surfaces of cryptographic size and near-prime order.}

\section{Introduction}\label{sec:intro}

The security of elliptic curve cryptography depends on the difficulty of the elliptic curve discrete log problem (ECDLP). Given an elliptic curve $E$ over $\FF_p$, a cyclic subgroup of $E(\FF_p)$ generated by the point $P$, and a point $Q \in \free{P}$, the ECDLP asks to find an integer $k$ such that $Q = kP$. The fastest known generic algorithm to solve the ECDLP on an elliptic curve is Pollard's rho algorithm, which has an expected runtime of $\tilde{O}(\sqrt{p})$ \cite[Ch.~3.6.3]{menezes1997handbook}.

It is possible to transfer the ECDLP between curves via isogenies. If $\varphi:E \to E'$ is an isogeny\footnote{Unless otherwise noted, by isogeny we mean $\FF_p$-rational isogeny.} that restricts to an isomorphism $\free{P} \to \free{\varphi(P)}$, then $Q = kP$ if and only if $\varphi(Q) = k\varphi(P)$. This reduction is useful for solving the ECDLP if the time it takes to compute $\varphi(Q)$ and $\varphi(P)$, as well as to solve the ECDLP on $E'$, is less than the time it takes to solve the ECDLP on $E$.

Concern about isogeny based attacks is partially motivated by the Gaudry-Hess-Smart (GHS) attack \cite{gaudry2000algorithm}. Over certain extension fields\footnote{The attack described in \cite{menezes2006cryptographic} only applies to fields of the form $\FF_2^{3\ell}$ with $53 \leq \ell \leq 200$.}, Menezes and Teske in \cite[Sec.~7]{menezes2006cryptographic} used the generalized GHS attack to show that there is a non-negligible proportion of ``weak'' curves, for which the ECDLP can be solved in significantly less time than it takes Pollard's rho. Under some reasonable assumptions, given a random elliptic curve $E$ over such a field, one can find a chain of efficiently computable isogenies from $E$ to a weak curve.

In \cite[Sec.~11, Ex.~5]{koblitz2011elliptic}, Koblitz, Koblitz, and Menezes observed that it is possible to use the complex multiplication (CM) method to construct elliptic curves $E/\FF_p$ whose isogeny class is large ($\approx \sqrt{p}$), but contains no curves (besides $E$ itself) whose \emph{conductor gap} with $E$ is small. The conductor gap between two curves measures the computational complexity in computing an isogeny between them. The curve $E$ is called \emph{isolated} because there are no other curves $E'$ for which constructing an isogeny between $E$ and $E'$ is computationally feasible.

So far we have only mentioned elliptic curves, but the same ideas carry over to abelian surfaces.
In \cite{wenhan2012isolated}, Wang gave a construction for isolated abelian surfaces that is analogous to the one for curves given in \cite{koblitz2011elliptic}. Note that while these methods construct isolated varieties, they almost always have large isogeny classes.

In this paper, we focus on the special case of \emph{super-isolated} abelian varieties. We call an abelian variety over a finite field super-isolated if its isogeny class contains a single isomorphism class. For increased security and efficiency, we focus on varieties of prime or near-prime order defined over a prime field.

Our main contributions are as follows. First, we outline practical algorithms that search for super-isolated elliptic curves and abelian surfaces. Second, we prove that only two super-isolated surfaces of cryptographic size and near-prime order exist, see Examples~\ref{ex:iso-surface} and \ref{ex:iso-surface2}. Finally, we give some heuristics on the number of super-isolated varieties. Our results suggest that, unlike the case of surfaces, there are enough super-isolated elliptic curves of cryptographic size and prime order to use in cryptosystems that require ephemeral curves, such as \cite{miele2015efficient}.

The outline of the paper is as follows.
Section~\ref{sec:dim-1} focuses on elliptic curves. Some background and notation is given in Section~\ref{sec:dim-1-background}. In Section~\ref{sec:dim-1-alg}, we outline an algorithm to construct super-isolated elliptic curves of prime order over $\FF_p$ with $p$ of a given size. We heuristically estimate the number of such curves in Section~\ref{sec:count-super-isolated-curves}.
Section~\ref{sec:dim-2} focuses on surfaces. In Section~\ref{sec:dim-2-background}, we show that finding super-isolated surfaces reduces to finding \emph{super-isolated Weil numbers}, which are defined in that section. In Section~\ref{sec:dim-2-find-weil-numbers}, we outline an algorithm to search for super-isolated Weil numbers. We also prove the correctness and efficiency of the algorithm in the same section. Two examples of super-isolated surfaces of near-prime order and cryptographic size are given in Section~\ref{sec:dim-2-examples}. In Section~\ref{sec:dim-2-only-examples}, we prove these are the only such examples.

\subsubsection*{Acknowledgments}

I would like to thank my advisor Neal Koblitz for all of his inspiration and guidance while working on this paper. I would also like to acknowledge the support from my graduate student peers, for which I am especially grateful.

\section{Elliptic curves}\label{sec:dim-1}

\subsection{Background and notation}\label{sec:dim-1-background}

Let $p$ be a prime\footnote{Many results in this paper can be extended to varieties over arbitrary finite fields $\FF_q$, but we focus on the prime case because it is often more efficient in practice.}. For any $t \in \ZZ$, let $I(t)$ denote the set of isomorphism classes of elliptic curves $E/\FF_p$ such that $\#E(\FF_p) = p - t + 1$. A theorem of Tate says that the sets $I(t)$ are isogeny classes of elliptic curves over $\FF_p$, see \cite[Ch.~5]{silverman2009arithmetic}. The Hasse bound implies that $I(t)$ is empty when $t^2 > 4p$. An elliptic curve is \emph{ordinary} if $t \not\equiv 0 \mod{p}$.

\begin{defn}\label{def:ec-super-isolated}
	An elliptic curve $E/\FF_p$ is \emph{super-isolated} if there is only one isomorphism class in its isogeny class, i.e. $\#I(p + 1 - \#E(\FF_p)) = 1$.
\end{defn}

\begin{defn}\label{def:kronecker-class-number}
	Let $\sO$ be an order in a quadratic imaginary field, and let $\Delta$ be the discriminant of $\sO$. The \emph{Kronecker class number $H(\Delta)$ of $\Delta$} is defined to be
	\[
		H(\Delta) = \sum_{\sO' \supseteq \sO} h(\sO')
	\]
	where $h(\sO')$ denotes the class number of the order $\sO'$, and the sum is over all orders $\sO'$ of $\sO\otimes\QQ$ such that $\sO' \supseteq \sO$.
\end{defn}

\begin{thm}[{\cite[Thm.~4.6]{schoof1987nonsingular}}\ ]\label{thm:size-of-isogeny-class}
	If $t^2 < 4p$ and $t \not\equiv 0\mod{p}$, then
	\[
		\#I(t) = H(t^2 - 4p).
	\]
\end{thm}

\begin{rem}
	The Kronecker class number is defined differently in \cite[Defn.~2.1]{schoof1987nonsingular}, but the equivalence with the Definition~\ref{def:kronecker-class-number} is proved in \cite[Prop.~2.4]{schoof1987nonsingular}.
\end{rem}

\begin{rem}\label{rem:avoid-super-singular}
	If $t = p + 1 - \#E(\FF_p) \equiv 0 \mod{p}$, then $E$ is called \emph{supersingular}. The reason that we focus on ordinary curves is because the ECDLP on supersingular curves is vulnerable to the Menezes-Okamoto-Vanstone attack \cite{menezes1993reducing}. There do exist super-isolated supersingular curves. For example, $y^2 + y = x^3 + x$ is the only curve over $\FF_2$ with $5$ points. See \cite[Thm.~4.6, Pg.~194]{schoof1987nonsingular} for a detailed formula for $\#I(t)$ when $t \equiv 0 \mod{p}$. If $p \geq 5$ then any supersingular curve over $\FF_p$ will have an even number of points. Hence we may ignore the supersingular case because we are interested in curves with prime order.
\end{rem}

\subsection{Super-isolated elliptic curves of prime order}\label{sec:dim-1-alg}

In this section, we outline a simple method to search for super-isolated elliptic curves which have prime order. The reason for considering curves of prime order is that it increases the security and efficiency of the elliptic curve cryptosystem.

First we will use the results in Section~\ref{sec:dim-1-background} to give a simple characterization of super-isolated elliptic curves over prime fields.

\begin{cor}\label{cor:super-isolated-characterization}
	Let $E/\FF_p$ be an ordinary elliptic curve with trace $t = p + 1 - \#E(\FF_p)$. Then $E$ is super-isolated if and only if
	\[
		t^2 - 4p \in \left\{ -3,-4,-7,-8,-11,-19,-43,-67,-163 \right\}.
	\]
\end{cor}
\begin{proof}
	By Theorem~\ref{thm:size-of-isogeny-class}, $\#I(t) = 1$ if and only if $t^2 - 4p$ is the discriminant of the maximal order of a quadratic imaginary field with class number $1$. It is a well known theorem of Heegner and Stark that the numbers in the statement are precisely the discriminants of such fields \cite{stark1967complete}.
\end{proof}

\begin{rem}
	Super-isolated elliptic curves are rare in the sense that if we choose a prime $p$ at random, it is unlikely there exists such a curve over $\FF_p$. Let $\pi_{SI}(x)$ denote the number of primes $p < x$ such that there exists a super-isolated elliptic curve over $\FF_p$. Any such $p$ must be of the form $t^2+d/4$, where $-d$ is one of the numbers from Corollary~\ref{cor:super-isolated-characterization} and $t$ is an integer. This shows that $\pi_{SI}(x) = O(\sqrt{x})$.
\end{rem}

\begin{rem}
	Even when super-isolated curves exist over $\FF_p$, such curves are rare in the set of all curves over $\FF_p$. There are $2p + O(1)$ distinct isomorphism classes of elliptic curves over $\FF_p$, but at most $18$ are super-isolated. The number $18$ is a rough overestimate that comes from the $9$ values in Corollary~\ref{cor:super-isolated-characterization}, and then multiplying by 2 to account for quadratic twists. It is not hard to show that other twists will not be super-isolated.
\end{rem}

\begin{rem}\label{rem:finding-curves-in-weil-p-notation}
	Another way to view the condition in Corollary~\ref{cor:super-isolated-characterization} is as follows. Let $K = \QQ(\sqrt{-d})$ be an imaginary quadratic field with class number $1$ and discriminant $-d$. Then we are searching for algebraic integers $\pi \in \sO_K$ of the form $\pi = (t + \sqrt{-d})/2$ such that $\pi\overline{\pi} = (t^2 + d)/4 = p$ is prime and $\ZZ[\pi] = \sO_K$.
\end{rem}

Suppose that $p$, $t$, and $d = t^2 - 4p$ satisfy the condition in Corollary~\ref{cor:super-isolated-characterization}. The fact that $p = (t^2+d)/4 \in \ZZ$ implies that $t \equiv d \mod{2}$. So we may replace $t$ with $2x$ or $2x + 1$ depending on $d\mod{2}$. Then $p$ and $N = p + 1 - t$ can be written as the following integral polynomials:

\begin{equation}\label{eq:polynomials-p-N-for-curves}
	p = \begin{cases}
		x^2 + \frac{d}{4} &\text{if $-d \equiv 0 \mod{4}$}
		\\
		x^2 + x + \frac{d + 1}{4} &\text{if $-d \equiv 1 \mod{4}$}
	\end{cases}
	,\quad
	N = \begin{cases}
		(x - 1)^2 + \frac{d}{4} &\text{if $-d \equiv 0 \mod{4}$}
		\\
		x^2 - x + \frac{d + 1}{4} &\text{if $-d \equiv 1 \mod{4}$}
	\end{cases}
\end{equation}

We are interested in values of $x$ such that $p$ and $N$ are simultaneously prime. Two necessary conditions for $p$ and $N$ to be simultaneously prime infinitely often are:
\begin{enumerate}[label=(\roman*)]
	\item\label{cond-p,N-irreducible} $p$ and $N$ are irreducible over $\ZZ[x]$.
	\item\label{cond-p,N-coprime} $\gcd_{a\in\ZZ} p(a)N(a) = 1$.
\end{enumerate}
It is clear that condition~\ref{cond-p,N-irreducible} is satisfied for all values of $d$. From Table~\ref{tbl:d-values-so-p-N-can-be-prime} below, condition~\ref{cond-p,N-coprime} holds for $d \in \{3,19,43,67,163\}$ (this can be checked using only a few consecutive values of $a$ \cite[Ex.~3.i, Pg.~19]{cahen1997integer}).

We now give a simple description of our search method.
\begin{enumerate}
	\item Choose $d \in \{3,19,43,67,163\}$ and let $p,N$ be as in Table~\ref{tbl:C-values-for-pN}.
	\item Choose random integers $x$ in a predetermined range until $p(x)$ and $N(x)$ are both prime.
	\item Use the CM method to recover a curve $E/\FF_p$ with $N$ points, see \cite[Ch.~18.1]{cohen2006handbook}.
\end{enumerate}

\begin{ex}\label{ex:super-isolated-curve-try-1}
	Let $d = 3$ and $x = 321438704914423479101766132343967029098$. Then $p = p(x)$ and $N = N(x)$ are both $256$-bit primes. The curve $E/\FF_p$ given by $y^2 = x^3 + 244944$ satisfies $\#E(\FF_p) = N$. This value of $x$ was found by a {\tt Sage} \cite{sage} program that randomly sampled integers from the interval $[0,2^{128}]$.
\end{ex}

\begin{ex}\label{ex:super-isolated-curve-low-hamming-weight}
	Let $d = 3$ and $x = 2^{127} + 13906$. Then $p(x)$ and $N(x)$ are $255$-bit primes. Moreover, their binary representations have a Hamming weight of $24$ and $27$ respectively. The CM method gives the curve $y^2 = x^3 + 279936$. Even though our search method does not have full control over the prime $p$, it is still possible to find primes with certain desirable properties, such as a low Hamming weight.
\end{ex}

\begin{table}[!htb]
	\centering
	\begin{tabular}{c|c|c|c}
		$-d$ & $p(x)$ & $N(x)$ & $\gcd_{a\in\ZZ} p(a)N(a)$ \\\hline
		3 & $x^2 + x + 1$ & $x^2 - x + 1$ & 1 \\
		4 & $x^2 + 1$ & $x^2 - 2x + 2$ & 2 \\
		8 & $x^2 + 2$ & $x^2 - 2x + 3$ & 6 \\
		7 & $x^2 + x + 2$ & $x^2 - x + 2$ & 4 \\
		11 & $x^2 + x + 3$ & $x^2 - x + 3$ & 3 \\
		19 & $x^2 + x + 5$ & $x^2 - x + 5$ & 1 \\
		43 & $x^2 + x + 11$ & $x^2 - x + 11$ & 1 \\
		67 & $x^2 + x + 17$ & $x^2 - x + 17$ & 1 \\
		163 & $x^2 + x + 41$ & $x^2 - x + 41$ & 1 \\
	\end{tabular}
	\caption{$\gcd_{a\in\ZZ}p(a)N(a)$ for values of $d$.}
	\label{tbl:d-values-so-p-N-can-be-prime}
\end{table}

\subsection{Estimating the number of super-isolated curves of prime order}\label{sec:count-super-isolated-curves}

In various applications, it is important to have some degree of randomness in the parameter selection. For example, a cryptosystem may require a distinct curve for each user, or use ephemeral keys such as in \cite{miele2015efficient}. In this section, we estimate the number of super-isolated elliptic curves of prime order, as a way to measure the randomness in the selection of such a curve. We also give some numerical evidence supporting our estimates.

The Bateman-Horn conjecture \cite{bateman1962heuristic} implies that if $p(x)$ and $N(x)$ are irreducible and satisfy $\gcd_{a\in\ZZ}p(a)N(a) = 1$, then the number of $x$, with $0 \leq x \leq M$, such that $p(x)$ and $N(x)$ are simultaneously prime is asymptotic to $\frac{C}{4}\int_2^M 1/\log^{2}(t) dt$ for a computable constant $C$. It is clear that $p(x)$ and $N(x)$ are irreducible, and we saw in Table~\ref{tbl:d-values-so-p-N-can-be-prime} the values of $d$ such that the second property holds. For each such $d$, Table~\ref{tbl:C-values-for-pN} gives an approximation of the constant $C$.

\begin{table}[!htb]
	\centering
	\begin{tabular}{c|c}
		$-d$ & $C$ \\\hline
		$-3$ & $\approx 2.9$ \\
		$-19$ & $\approx 3.0$ \\
		$-43$ & $\approx 10.6$ \\
		$-67$ & $\approx 17.5$ \\
		$-163$ & $\approx 44.8$
	\end{tabular}
	\caption{An approximation to the Bateman-Horn constant $C$.}
	\label{tbl:C-values-for-pN}
\end{table}

\begin{ex}\label{ex:avg-num-tries-for-super-isolated-curve}
	We ran $10000$ iterations of the search in Example~\ref{ex:super-isolated-curve-try-1}. The average number of $x$'s sampled until $p(x)$ and $N(x)$ were both prime, was $10312$. The heuristics above imply that the expected number of $x$'s that need to be sampled is
	\[
		\left(\frac{2.9}{4}\frac{1}{2^{128}}\int_2^{2^{128}} \frac{1}{\log^2 t} \ dt\right)^{-1} \approx 10610.
	\]
	The percent difference between the observed and expected is $-0.028$.
%
\end{ex}

Combining the heuristics above, we expect that the number of $x$ with $0 \leq x \leq M$ such that $p(x)$ and $N(x)$ are prime for some $d\in\{3,19,43,67,163\}$ is approximately
\[
	19.7
	\cdot
	\int_2^M\frac{1}{\log^2 t} \ dt.
\]

Since $p(x)$ has degree $2$, we can estimate the number of curves over $\FF_p$ with $p \leq M$ by choosing $x$ in the range $0 \leq x \leq \sqrt{M}$. Combined with above, we have the following estimate for the number of curves.
\begin{heuristic}
	The number of super-isolated elliptic curves of prime order over $\FF_p$ with $p \leq M$ is approximately
	\[
		19.7\int_2^{\sqrt{M}}\frac{1}{\log^2 t} \ dt.
	\]
\end{heuristic}

\section{Abelian surfaces}\label{sec:dim-2}

\subsection{Super-isolated Weil numbers}\label{sec:dim-2-background}

We define \emph{super-isolated} for an abelian variety as we did for an elliptic curve: an abelian variety whose isogeny class contains only one isomorphism class. Recall that finding a super-isolated elliptic curve over $\FF_p$ is equivalent to finding an algebraic integer $\pi$ in an imaginary quadratic field $K$ of class number $1$ such that $\pi\overline{\pi} = p$ and $\ZZ[\pi] = \sO_K$ (see Remark~\ref{rem:finding-curves-in-weil-p-notation}). The general situation is similar, only we replace $K$ with a CM field (defined below), and $\ZZ[\pi]$ with $\ZZ[\pi,\overline{\pi}]$.

\begin{defn}\label{def:cm-field}
	A number field $K$ is a \emph{complex multiplication field}, or \emph{CM field}, if $K$ is a quadratic imaginary extension of a totally real field $F$. CM fields have a unique non-trivial automorphism fixing $F$, which we denote by $\alpha \mapsto \overline{\alpha}$ and refer to as complex conjugation.
\end{defn}

\begin{defn}\label{def:weil-number}
	For any $n \in \ZZ$, a \emph{Weil $n$-number} is an algebraic integer that has absolute value $\sqrt{n}$ under every embedding to $\CC$. A \emph{Weil number} is a Weil $n$-number for some $n$. If $K$ is a CM field, then $\alpha \in \sO_K$ is a Weil number if and only if $\alpha\overline{\alpha} \in \ZZ$. It can be shown that if $\pi$ is a Weil $p$-number for a prime $p$, then either $\QQ(\pi)$ is a CM field or $\pi = \pm\sqrt{p}$.
\end{defn}

Let $A/\FF_p$ be a simple\footnote{In this paper we use simple to mean simple over the base field. Other sources sometimes use the term to mean simple over the algebraic closure.} abelian variety, and let $f$ be the characteristic polynomial of the Frobenius endomorphism of $A$. It is well known that $\#A(\FF_p) = f(1)$ and $f = h^e$ where $h$ is irreducible and $e$ is some integer. Moreover, any root $\pi$ of $f$ is a Weil $p$-number and $2\dim A = e[\QQ(\pi):\QQ]$ \cite[Thm~8]{waterhouse1971abelian}. For cryptographic reasons, we are interested in varieties with prime or near-prime order, so we will mainly focus on the case where $e = 1$.

\begin{thm}\label{thm:general-criteria-to-be-super-isolated}
	Let $A$ be a simple abelian variety over $\FF_p$, $\pi$ be a root of the characteristic polynomial of the Frobenius endomorphism, and $K = \QQ(\pi)$. Assume that $\pi \neq \pm\sqrt{p}$. Then $A$ is super-isolated if and only if $\ZZ[\pi,\overline{\pi}] = \sO_K$ and $K$ has class number $1$.
\end{thm}
\begin{proof}
	By \cite[Thm.~3.5]{waterhouse1969abelian}, the endomorphism ring of any variety isogenous to $A$ is an order\footnote{The statement of \cite[Thm.~3.5]{waterhouse1969abelian} refers to an order in $\End_{\FF_p} A\otimes\QQ$, but this is the same as $K$ since the base field is prime, see \cite[Ch.~2]{waterhouse1969abelian}.} in $\sO_K$ containing $\ZZ[\pi,\overline{\pi}]$. Because the base field is $\FF_p$ and $\pi \neq \pm\sqrt{p}$, the converse holds as well \cite[Thm.~6.1]{waterhouse1969abelian}. That is, every order of $\sO_K$ containing $\ZZ[\pi,\overline{\pi}]$ is the endomorphism ring of some variety isogenous to $A$. We call the set of varieties isogenous to $A$ with endomorphism ring $R$ the \emph{endomorphism class} of $R$. So there is exactly one endomorphism class if and only if $\ZZ[\pi,\overline{\pi}] = \sO_K$.
	The proof of \cite[Thm.~6.1]{waterhouse1969abelian} shows that the number of isomorphism classes in the endomorphism class of $\sO_K$ is equal to the class number of $K$.
	Therefore, the entire isogeny class of $A$ contains a single isomorphism class if and only if $\ZZ[\pi,\overline{\pi}] = \sO_K$ and $K$ has class number~$1$.
\end{proof}

\begin{defn}\label{def:super-isolated-weil-number}
	A \emph{super-isolated Weil number} is a Weil number $\pi$ such that $K = \QQ(\pi)$ has class number $1$ and $\sO_K = \ZZ[\pi,\overline{\pi}]$.
\end{defn}

In this section we are mainly interested in surfaces. The reason for not considering higher dimensional abelian varieties is that the discrete log problem on jacobians\footnote{
	Cryptosystems usually use jacobians of hyperelliptic curves rather than arbitrary varieties because they provide efficient representations necessary for practical use \cite{koblitz1989hyperelliptic}.
} of curves of genus $\geq 3$ can be solved faster than on comparably sized jacobians of curves of genus $\leq 2$ \cite{enge2002computing,gaudry2007double,smith2008isogenies}. This means that we would need to use a larger key size in order to achieve comparable security. Hence varieties of dimension $\geq 3$ are less efficient in practice. Both genus $2$ and $1$ are still considered for cryptographic use and have comparable efficiency \cite{bernstein2014kummer}.

\begin{rem}
	Another reason for focusing on curves and surfaces is that we do not expect many super-isolated varieties of dimension at least $3$. In \cite{effective1974stark}, Stark remarked that it is reasonable to believe that there are only finitely many CM fields of class number $1$. This would imply that there is a finite list of fields which admit Weil numbers that could correspond to super-isolated varieties of near-prime order.
\end{rem}

The following corollary specializes Theorem~\ref{thm:general-criteria-to-be-super-isolated} to surfaces with $e=1$.

\begin{cor}\label{cor:surfaces-super-isolated-condition}
	Let $A$ be an abelian surface over $\FF_p$. Assume that the characteristic polynomial $f$ of the Frobenius endomorphism of $A$ is irreducible. Then $A$ is super-isolated if and only if the roots of $f$ are super-isolated Weil $p$-numbers.
\end{cor}
\begin{proof}
	This follows from Theorem~\ref{thm:general-criteria-to-be-super-isolated} after noting that, because $f$ is irreducible, $A$ is simple and $\pm\sqrt{p}$ can not be roots of $f$.
\end{proof}

Therefore, in order to find super-isolated surfaces of near-prime order (note that near-prime order implies the hypothesis in Corollary~\ref{cor:surfaces-super-isolated-condition}), it is sufficient to find all super-isolated Weil numbers $\pi$, such that $\pi\overline{\pi}$ is prime and $\QQ(\pi)$ has degree $4$. There are 91 quartic CM fields of class number $1$, and they can be found in the literature \cite{yamamura1994determination,louboutin1994determination}. By \cite[Cor.~2.10]{maisner2002abelian}, if $\pi$ is a Weil $p$-number whose minimal polynomial $f$ has degree $4$, then there is a simple abelian surface over $\FF_p$ such that $f$ is the characteristic polynomial of the Frobenius endomorphism of $A$. This is a special case of a theorem of Honda, which shows that every Weil $p$-number is a root of the characteristic polynomial of the Frobenius endomorphism of some simple abelian variety over $\FF_p$ \cite{honda1968isogeny}. One can recover a representative of the isogeny class of $A$ from $\pi$ using the two dimensional analogue of the CM method \cite[Ch.~18]{cohen2006handbook}.



\subsection{Search algorithm}\label{sec:dim-2-find-weil-numbers}

In this section we describe an efficient algorithm for enumerating all super-isolated Weil numbers in a given field up to a certain bound. For the rest of the paper, unless otherwise stated, we will only consider super-isolated Weil numbers $\pi$ such that $\QQ(\pi)$ has degree $4$ over $\QQ$ and $\pi\overline{\pi}$ is prime.

\begin{rem}
	Our methods are motivated by those Wang used in \cite{wenhan2012isolated} to parameterize \emph{isolated abelian surfaces}, which are analogues of the isolated elliptic curves described in Section~\ref{sec:intro}.
\end{rem}

\begin{rem}
	A naive algorithm to find super-isolated Weil $p$-numbers is as follows. Fix a quartic CM field $K$ with class number $1$. For each prime $p$ less than a certain bound, find all possible solutions $\pi$ in $\sO_K$ to the relative norm equation $\pi\overline{\pi} = p$. This can be done using standard algorithms, see \cite[Ch.~7.5.4]{cohen2000advanced}. For each solution, check if $\ZZ[\pi,\overline{\pi}] = \sO_K$ by computing discriminants. This method is not practical because primes $p$ which admit super-isolated Weil $p$-numbers are rare.
\end{rem}

First, we will give an informal description of our algorithm. Let $K$ be a quartic CM field with class number $1$, and let $\{\alpha_1,\alpha_2,\alpha_3,\alpha_4\}$ be a basis for $\sO_K$. Then any $\pi \in \sO_K$ can be written as $\sum a_i\alpha_i$ for some $a_i\in\ZZ$. We will show that $\pi$ is a super-isolated Weil number if and only if the $a_i$ satisfy the following properties:
\begin{enumerate}[label=(\roman*)]
	\item The condition that $\pi\overline{\pi} \in \ZZ$ is equivalent to $\po(a_1,a_2,a_3,a_4) = 0$, where $\po$ is the polynomial in Equation~\ref{eq:p0} below.
	\item If (i) holds, then the condition that $\ZZ[\pi,\overline{\pi}] = \sO_K$ is equivalent to the equations
	\[
		\fl(a_1,a_2,a_3,a_4) = \pm 1
		\quad\text{ and }\quad
		\fll(a_1,a_2,a_3,a_4) = \pm 1,
	\]
	where $\fl$ and $\fll$ are the polynomials in Equations~(\ref{eq:f1},\ref{eq:f2}).
	\item The condition that $\pi\overline{\pi}$ is prime is equivalent to $\px(a_1,a_2,a_3,a_4)$ being prime, where $\px$ is the polynomial in Equation~\ref{eq:p} below.
\end{enumerate}
These equivalences are shown in the proof of Theorem~\ref{thm:g=2-algorithm-works} below. Moreover, we will also show that if $\{\alpha_1,\alpha_2,\alpha_3,\alpha_4\}$ are chosen in a certain way, then finding solutions to the equations $\po=0$, $\fl=\pm1$, $\fll=\pm1$ essentially reduces to an instance of Pell's equation. Our algorithm starts by choosing such a basis, and proceeds to enumerate tuples $(a_1,a_2,a_3,a_4)$ satisfying the conditions above.

\subsubsection{The algorithm}\label{sec:dim-2-alg}

The algorithm outlined below enumerates super-isolated Weil numbers in a certain field.

\begin{enumerate}
	\item\label{step:choose-field} Choose a quartic CM field $K$ of class number $1$, and let $F$ be the real quadratic subfield. Let $\Delta_K,\Delta_F$ denote the respective discriminants.
	
	\item\label{step:choose-basis} Choose a basis $\{\alpha_1,\alpha_2,\alpha_3,\alpha_4\}$ of $\sO_K$ such that $\alpha_1 = 1$ and $\{\alpha_1,\alpha_2\}$ form a basis for $\sO_F$.
	
	\item\label{step:choose-embeddings} Choose non-conjugate embeddings $\phi_1,\phi_2: K \hookrightarrow \CC$.
	
	\item\label{step:compute-polynomials} Compute the coefficients of the following polynomials:
	\begin{align}
		\fl &= \frac{1}{\sqrt{\Delta_F}} \sum_{i=1}^4 (\phi_1(\alpha_i + \overline{\alpha}_i) - \phi_2(\alpha_i + \overline{\alpha}_i))x_i
		\label{eq:f1}
		\\
		\fll &= \frac{\Delta_F}{\sqrt{\Delta_K}} \sum_{1 \leq i,j \leq 4} \phi_1(\alpha_i - \overline{\alpha}_i)\phi_2(\alpha_j - \overline{\alpha}_j)x_ix_j
		\label{eq:f2}
		\\
		\px &= \frac{1}{2} \sum_{1 \leq i,j \leq 4} \left( \phi_1(\alpha_i\overline{\alpha}_j) + \phi_2(\alpha_i\overline{\alpha}_j) \right)x_ix_j
		\label{eq:p}
		\\
		\po &= \frac{1}{2\sqrt{\Delta_F}}\sum_{1 \leq i,j \leq 4} \left( \phi_1(\alpha_i\overline{\alpha}_j) - \phi_2(\alpha_i\overline{\alpha}_j) \right)x_ix_j
		\label{eq:p0}
	\end{align}
	By Lemma~\ref{lem:f1,f2,p,p0-coefficients} below, $\fl \in \ZZ[x_2,x_3,x_4]$, $\fll \in \ZZ[x_3,x_4]$, and $\px,\po \in \frac{1}{2}\ZZ[x_1,x_2,x_3,x_4]$.
	
	\item\label{step:a3a4} Enumerate solutions $x_3=a_3$, $x_4=a_4$ to the equation
	\begin{equation}\label{cond:f2=pm1}
		\fll(x_3,x_4) = \pm 1
	\end{equation}
	up to a given bound. We do this as follows.

	\begin{enumerate}[label=\theenumi.\arabic*]
		\item\label{step:choose-generator} By Lemma~\ref{lem:disc-of-f2} below, we may write $\fll = ax_3^2 + bx_3x_4 + cx_4^2$ for $a,b,c \in \ZZ$ with $b^2 - 4ac = \Delta_F$. A straight-forward calculation shows that the $\ZZ$-module $I = a\ZZ + \frac{b + \sqrt{\Delta_F}}{2}\ZZ$ is an ideal of $\sO_F$.
		Here we are abusing notation by writing $\sqrt{\Delta_F}$ as an element of $F$.
		Since $K$ is a quartic CM field with class number $1$, $F$ has class number $1$. So we can choose a principal generator $\gamma$ for $I$. Let $\epsilon$ be a fundamental unit for $F$. One can find both $\gamma$ and $\epsilon$ using standard algorithms, see \cite[Ch.~4-5]{cohen1993course}.
		\item\label{step:choose-i} For each $i \in \ZZ$, with $|i|$ less than a predetermined bound, compute $\sigma = \pm\epsilon^i\gamma$ for each choice of sign.
		\item\label{step:compute-a3a4} For each $\sigma$, find a pair $a_3,a_4 \in \QQ$ such that $a_3a + a_4\frac{b+\sqrt{\Delta_F}}{2} = \sigma$.
	\end{enumerate}
	
	\item\label{step:a2} For each pair $(a_3,a_4)$, find all $a_2$ such that $x_2=a_2$, $x_3=a_3$, $x_4=a_4$ is a solution to
	\begin{equation}\label{cond:f1=pm1}
		\fl(x_2,x_3,x_4) = \pm 1.
	\end{equation}
	We can find $a_2$ as follows. By the choice of basis, the coefficient of $x_2$ in $\fl$ is non-zero. Thus, there are two possibilities for $a_2 \in \QQ$, and they are each given by linear polynomials in $a_3$, $a_4$.
	
	\item\label{step:a1} For each tuple $(a_2,a_3,a_4)$, find all $a_1$ such that $x_1=a_1$, $x_2=a_2$, $x_3=a_3$, $x_4=a_4$ is a solution to \begin{equation}\label{cond:p0=0}
		\po(x_1,x_2,x_3,x_4) = 0.
	\end{equation}
	This can be done as follows. A straightforward computation, using the fact that $\alpha_1 = 1$, shows that $2\po = g + \fl x_1$ for some polynomial $g\in \ZZ[x_2,x_3,x_4]$. Since $\fl(a_2,a_3,a_4) = \pm 1$, we have $a_1 = \mp g(a_2,a_3,a_4)$.
	
	\item\label{step:check-output} For each tuple $(a_1,a_2,a_3,a_4)$, if every $a_i$ is integral and $\px(a_1,a_2,a_3,a_4)$ is prime, then output
	\[
		\pi = a_1\alpha_1+ a_2\alpha_2 + a_3\alpha_3 + a_4\alpha_4.
	\]
\end{enumerate}

\subsubsection{Correctness}\label{sec:dim-2-alg-correct}

In this section, we will prove the correctness of the algorithm of Section~\ref{sec:dim-2-alg}. By correctness, we mean that if the algorithm outputs $\pi$, then $\pi$ is a super-isolated Weil number. Conversely, if $\pi$ is a super-isolated Weil number in $K$, then, given a large enough bound, the algorithm will eventually output $\pi$.

\begin{rem}
	This section is solely focused on the correctness of the algorithm. For a discussion of the efficiency, see Section~\ref{sec:dim-2-alg-runtime}.
\end{rem}

Our proof of correctness involves several computations, which have been broken down into several lemmas. The main idea is to find explicit polynomials representing the index of $\ZZ[\pi,\pi]$ in $\sO_K$ and the value of $\pi\overline{\pi}$, both with respect to the basis $\{\alpha_1,\alpha_2,\alpha_3,\alpha_4\}$.

First, we will prove Lemmas~\ref{lem:f1,f2,p,p0-coefficients} and \ref{lem:disc-of-f2}, which were used in the description of the polynomials $\fl,\fll,\px,\po$ in the algorithm in Section~\ref{sec:dim-2-alg}. To prove these lemmas, we start with some facts from algebraic number theory.

\begin{lem}\label{lem:trace-to-F-in-1/2Z}
	Let $K$ be a quartic CM field, $F$ be the quadratic real subfield of $K$, and $\phi_1,\phi_2$ be non-conjugate embeddings $K \hookrightarrow \CC$. If $\gamma \in \sO_K$, then
	\[
		\phi_1(\gamma + \overline{\gamma}) + \phi_2(\gamma + \overline{\gamma}) \in \ZZ
	\]
	and
	\[
		\phi_1(\gamma + \overline{\gamma}) - \phi_2(\gamma + \overline{\gamma}) \in \sqrt{\Delta_F}\ZZ.
	\]
\end{lem}
\begin{proof}
	Note that $\gamma + \overline{\gamma} \in \sO_F$, so $\phi_1(\gamma + \overline{\gamma})$ can be written as $a + b\sqrt{\Delta_F}$ for some $a,b \in \frac{1}{2}\ZZ$. Because $\phi_1(\sqrt{\Delta_F}) = -\phi_2(\sqrt{\Delta_F})$, the claim follows from noticing that
	\begin{align*}
		\phi_1(\gamma + \overline{\gamma}) + \phi_2(\gamma + \overline{\gamma}) &= 2a
		\\
		\phi_1(\gamma + \overline{\gamma}) - \phi_2(\gamma + \overline{\gamma}) &= 2b\sqrt{\Delta_F}.
	\end{align*}
\end{proof}

\begin{lem}\label{lem:imag-part-type-norm-in-deltaKF-Z}
	Let $K$ be a quartic CM field with maximal totally real subfield $F$. If $\gamma \in \sO_K$ and $\phi_1,\phi_2$ are any non-conjugate pair of embeddings $K \hookrightarrow \CC$, then
	\[
		\phi_1(\gamma - \overline{\gamma})\phi_2(\gamma - \overline{\gamma}) \in \frac{\sqrt{\Delta_K}}{\Delta_F} \cdot \ZZ.
	\]
\end{lem}
\begin{proof}
	Let $\delta_{K/F}(\alpha)$ denote the relative different for any $\alpha \in \sO_K$. Because $K/F$ is a quadratic imaginary extension, we have that $\delta_{K/F}(\alpha) = \alpha - \overline{\alpha}$.
	
	We may assume $K = F(\gamma)$ otherwise the claim is trivial as $\phi_1(\gamma - \overline{\gamma})\phi_2(\gamma - \overline{\gamma}) = 0$. From the proof of \cite[Thm.~III.2.5, Pg.~198]{neukrich1999algebraic},
	\[
		\delta_{K/F}(\gamma)\sO_K = \mfrak{f}_{\sO_F[\gamma]}\script{D}_{K/F}
	\]
	where $\mfrak{f}_{\sO_F[\gamma]} = \{\alpha \in K \colon \alpha\sO_K \subseteq \sO_F[\gamma]\}$ is the conductor of the order $\sO_F[\gamma]$ in $\sO_K$ and $\script{D}_{K/F}$ is the relative different of the extension $K/F$.
	
	Note that $\gamma + \overline{\gamma} \in \sO_F$ implies that $\overline{\gamma} \in \sO_F[\gamma]$, hence $\sO_F[\gamma]$ is invariant under conjugation. It follows that $\mfrak{f}_{\sO_F[\gamma]}$ is invariant under conjugation,
	so we may write $\mfrak{f}_{\sO_F[\gamma]} = I \cdot \sO_K$ for some ideal $I \subseteq \sO_F$.
	Then
	\begin{align}\label{eq:square-of-gamma-minus-gammabar}
		\left(\phi_1(\gamma - \overline{\gamma})\phi_2(\gamma - \overline{\gamma})\right)^2
		&= N_{K/\QQ} \left(\gamma - \overline{\gamma}\right)
		\nonumber
		\\
		&= N_{K/\QQ} \left(\delta_{K/F}(\gamma)\right)
		\nonumber
		\\
		&= N_{F/\QQ}(I)^2 \cdot N_{K/\QQ}(\script{D}_{K/F})
	\end{align}
	The different and discriminant are related by the formula \cite[Cor.~III.2.10, Pg.~197]{neukrich1999algebraic}
	\begin{equation}\label{eq:relative-disc-and-diff}
		\Delta_K = \Delta_{F}^2N_{K/\QQ}\left(\script{D}_{K/F}\right).
	\end{equation}
	The claim follows from combining Equation~\ref{eq:square-of-gamma-minus-gammabar} and Equation~\ref{eq:relative-disc-and-diff} and taking square roots.
\end{proof}

\begin{lem}\label{lem:f1,f2,p,p0-coefficients}
	Let $\fl,\fll,\px,\po$ be the polynomials from Equations~(\ref{eq:f1})-(\ref{eq:p0}). Then
	\begin{enumerate}[label=(\roman*)]
		\item $\fl \in \ZZ[x_2,x_3,x_4]$
		\item $\fll \in \ZZ[x_3,x_4]$
		\item $\px \in \frac{1}{2}\ZZ[x_1,x_2,x_3,x_4]$
		\item $\po \in \frac{1}{2}\ZZ[x_1,x_2,x_3,x_4]$.
	\end{enumerate}
\end{lem}
\begin{proof}
	It is straightforward from the definition of $\fl,\fll$ and the choice of basis $\{\alpha_1,\alpha_2,\alpha_3,\alpha_4\}$ that $\fl \in \CC[x_2,x_3,x_4]$ and $\fll \in \CC[x_3,x_4]$. So it remains to check the domain of the coefficients.
	
	The claims for $\fl$, $\px$, and $\po$ all follow directly from Lemma~\ref{lem:trace-to-F-in-1/2Z}. For $\fll$, note that the coefficient of $x_ix_j$ is $\phi_1(\delta_i)\phi_2(\delta_j) + \phi_1(\delta_j)\phi_2(\delta_i)$, where $\delta_i = \alpha_i - \overline{\alpha}_i$. The claim for $\fll$ follows from Lemma~\ref{lem:imag-part-type-norm-in-deltaKF-Z} and the fact that
	\[
		\phi_1(\delta_i)\phi_2(\delta_j) + \phi_1(\delta_j)\phi_2(\delta_i)
		=
		\phi_1(\delta_i+\delta_j)\phi_2(\delta_i+\delta_j)
		-
		\phi_1(\delta_i)\phi_2(\delta_i)
		-
		\phi_1(\delta_j)\phi_2(\delta_j).
	\]
\end{proof}

\begin{lem}\label{lem:disc-of-f2}
	$\fll$ is a integral bilinear quadratic form in $x_3,x_4$ with discriminant $\Delta_F$.
\end{lem}
\begin{proof}
	It is clear from the definition that $\fll$ is a homogeneous polynomial of degree $2$, and by Lemma~\ref{lem:f1,f2,p,p0-coefficients}, $\fll \in \ZZ[x_3,x_4]$. So it remains to calculate the discriminant.
	
	Let $\delta_i = \alpha_i - \overline{\alpha}_i$. By definition,
	\begin{align*}
		\disc \fll
		&= \frac{\Delta_F^2}{\Delta_K}\left(
		\left(\phi_1(\delta_3)\phi_2(\delta_4) + \phi_2(\delta_3)\phi_1(\delta_4)\right)^2
		-
		4\phi_1(\delta_3)\phi_2(\delta_3)\phi_1(\delta_4)\phi_2(\delta_4)
		\right)
		\\
		&=
		\frac{\Delta_F^2}{\Delta_K}\left(\phi_1(\delta_3)\phi_2(\delta_4) - \phi_2(\delta_3)\phi_1(\delta_4)\right)^2.
	\end{align*}
	Now we compute
	\begin{align*}
		\Delta_K
		&= \det\begin{pmatrix}
		1 & \phi_1(\alpha_2) & \phi_1(\alpha_3) & \phi_1(\alpha_4)
		\\
		1 & \phi_1(\alpha_2) & \phi_1(\overline{\alpha_3}) & \phi_1(\overline{\alpha}_4)
		\\
		1 & \phi_2(\alpha_2) & \phi_2(\alpha_3) & \phi_2(\alpha_4)
		\\
		1 & \phi_2(\alpha_2) & \phi_2(\overline{\alpha}_3) & \phi_2(\overline{\alpha}_4)
		\end{pmatrix}^2
		\\
		&=
		\left(\phi_1(\alpha_2) - \phi_2(\alpha_2)\right)^2\left(\phi_1(\delta_3)\phi_2(\delta_4) - \phi_2(\delta_3)\phi_1(\delta_4)\right)^2
		\\
		&=
		\Delta_F\left(\phi_1(\delta_3)\phi_2(\delta_4) - \phi_2(\delta_3)\phi_1(\delta_4)\right)^2.
	\end{align*}
\end{proof}

Next we prove the correctness of step~\ref{step:a3a4} using our previous lemmas.

\begin{lem}\label{lem:stepa3a4-works}
	If $(a_3,a_4)$ is outputted in step~\ref{step:a3a4} of the algorithm in Section~\ref{sec:dim-2-alg}, then $\fll(a_3,a_4) = \pm 1$. Moreover, if $x_3=a_3$, $x_4=a_4$ is an integral solution to $\fll(x_3,x_4) = \pm 1$, then, given a large enough bound, step~\ref{step:a3a4} will eventually output the pair $(a_3,a_4)$.
\end{lem}
\begin{proof}
	
	Following the notation from the algorithm, let $\fll(x_3,x_4) = ax_3^2 + bx_3x_4 + cx_4^2$. Recall from Lemma~\ref{lem:disc-of-f2} that $a,b,c\in\ZZ$ and $b^2 - 4ac = \Delta_F$. Note that $\{a, (b+\sqrt{\Delta_F})/2\}$ is a $\QQ$-basis for $F$. Here we are abusing notation by writing $\sqrt{\Delta_F}$ as an element of $F$. So we can write any $\sigma \in F$ as
	\[
		\sigma = ax + \frac{b + \sqrt{\Delta_F}}{2}y,
	\]
	for some $x,y \in \QQ$. Then the norm of $\sigma$ is
	\[
		\norm_{F/\QQ}(\sigma)
		= \left(ax + \frac{b + \sqrt{\Delta_F}}{2}y\right)\left(ax + \frac{b - \sqrt{\Delta_F}}{2}y\right)
		= a\left(ax^2 + bxy + cy^2\right) = a\fll(x,y).
	\]
	Therefore $\fll(x,y) = \pm 1$ if and only if the corresponding $\sigma$ has norm $\pm a$. Moreover, $x,y \in \ZZ$ if and only if $\sigma$ lies in the ideal $I = a\ZZ + (b+\sqrt{\Delta_F})/2\ZZ$ (one can show this is an ideal using the fact that $b^2 - 4ac = \Delta_F$).
	Because $\norm_{F/\QQ}(I) = |a|$,
	it follows that $x_3=x,x_4=y$ is an integral solution to $\fll(x_3,x_4) = \pm 1$ if and only if $\sigma\sO_F = I$.
	Therefore, we have a bijection between integral solutions to $\fll(x_3,x_4) = \pm 1$ and generators of $I$.
	
	The claim follows as steps~(\ref{step:choose-generator})-(\ref{step:compute-a3a4}) enumerate all generators $\sigma$ for the ideal $I$, and compute the associated integral solution to $\fll = \pm 1$.
\end{proof}

Now we will find an explicit $\ZZ$-basis for the order $\ZZ[\pi,\overline{\pi}]$. This will allow us to write down a formula for $\disc\ZZ[\pi,\overline{\pi}]$ in terms of the coefficients of $\pi$ with respect to the basis $\{\alpha_1,\alpha_2,\alpha_3,\alpha_4\}$ for $\sO_K$. We will use this formula to determine when $\ZZ[\pi,\overline{\pi}] = \sO_K$.

\begin{lem}\label{lem:1-pi-pibar-pi2-basis-for-Z-pi-pibar}
	Let $K$ be a quartic CM field and let $\pi \in \sO_K$ be an Weil $p$-number. Then $B = \free{1,\pi,\overline{\pi},\pi^2}$ generates $\ZZ[\pi,\overline{\pi}]$ as a $\ZZ$-module.
\end{lem}
\begin{proof}
	We will show that any power of $\pi$ or $\overline{\pi}$ is contained in $\Span(B)$. The claim will follow because $\pi\overline{\pi} \in \ZZ$, so any product $\pi^i\overline{\pi}^j$ can be rewritten as sums of powers of $\pi$ or $\overline{\pi}$. Because $K$ is a quartic extension, we only have to show $\pi^3,\overline{\pi}^2,\overline{\pi}^3 \in \Span(B)$.
	
	First we will show that $\overline{\pi}^2 \in \Span(B)$. Let $F$ be the real quadratic subfield of $K$. Note that $\pi + \overline{\pi} \in \sO_F$, so it has a characteristic polynomial in $F$ of the form $x^2 + ax + b$ for some $a,b\in\ZZ$. It follows that
	\begin{align}\label{eq:pibar-squared-in-span-B}
		\left(\pi + \overline{\pi}\right)^2 &= -a\left(\pi + \overline{\pi}\right) - b
		\nonumber
		\\
		\pi^2 + \overline{\pi}^2 &= -a\pi - a\overline{\pi} - b - 2p
		\nonumber
		\\
		\overline{\pi}^2 &= -\pi^2 - a\pi - a\overline{\pi} - b - 2p.
	\end{align}
	Now recall that the characteristic polynomial of $\pi$ in $K$ is of the form $x^4 - cx^3 + dx^2 - cpx + p^2$ for some $c,d \in \ZZ$. Using the fact that $\overline{\pi} = p/\pi$, this shows that
	\begin{align}
		0 &= \pi^4 - c\pi^3 + d\pi^2 - cp\pi + p^2
		\nonumber
		\\
		\pi^3 &= c\pi^2 - d\pi + cp - p\overline{\pi}
		\label{eq:pi-cubed-in-span-B}
		\\
		\overline{\pi}^3 &= c\overline{\pi}^2 + d\overline{\pi} + cp + p\pi.
		\label{eq:pibar-cubed-in-span-B}
	\end{align}
	It follows from Equations~(\ref{eq:pibar-squared-in-span-B})-(\ref{eq:pibar-cubed-in-span-B}) that $\overline{\pi}^2,\pi^3,\overline{\pi}^3 \in \Span(B)$.
\end{proof}

\begin{lem}\label{lem:disc-1-gamma-gammabar-gamma2}
	Let $K$ be a quartic CM field and let $\phi_1,\phi_2$ be non-conjugate embeddings $K \hookrightarrow \CC$. If $\gamma \in \sO_K$, then
	\[
		\disc(1,\gamma,\overline{\gamma},\gamma^2)
		=
		\left(\phi_1(\gamma + \overline{\gamma}) - \phi_2(\gamma - \overline{\gamma})\right)^4
		\left(\phi_1(\gamma - \overline{\gamma})\phi_2(\gamma - \overline{\gamma})\right)^2.
	\]
\end{lem}
\begin{proof}
	Let $\beta_1 = 1, \beta_2 = \gamma, \beta_3 = \overline{\gamma}, \beta_4 = \gamma^2$. Then $\disc(1,\gamma,\overline{\gamma},\gamma^2) = \det \tr\beta_i\beta_j$. Let $\gamma_i = \phi_i(\gamma)$. Because $K$ is a CM field, complex conjugation commutes with embeddings into $\CC$, so $\phi_i(\overline{\gamma}) = \overline{\gamma}_i$.
	Using this, we can compute $\tr\beta_i\beta_j$ in terms of $\gamma_1,\gamma_2$. For example:
	\begin{align*}
		\tr_{K/\QQ}\beta_3\beta_4 = \tr_{K/\QQ}\gamma^2\overline{\gamma} &= \gamma_1^2\overline{\gamma}_1 + \gamma_2^2\overline{\gamma}_2 + \gamma_1\overline{\gamma}_1^2 + \gamma_2\overline{\gamma}_2^2
	\end{align*}
	A straightforward computation shows that $\det \tr \beta_i\beta_j$, when viewed as a polynomial in the ring $\ZZ[\gamma_1,\gamma_2,\overline{\gamma}_1,\overline{\gamma}_2]$, factors into the desired form.
\end{proof}

We are now ready to prove the correctness of the algorithm.

\begin{thm}\label{thm:g=2-algorithm-works}
	If the algorithm of Section~\ref{sec:dim-2-alg} outputs $\pi$, then $\pi$ is a super-isolated Weil number. Moreover, for any fixed super-isolated Weil number $\pi$, if the algorithm is given $K = \QQ(\pi)$ and a large enough bound, then it will eventually output $\pi$.
\end{thm}
\begin{proof}
	We will use the same notation as in Section~\ref{sec:dim-2-alg}. Let $a_1,a_2,a_3,a_4 \in \ZZ$ and let $\pi = \sum a_i\alpha_i$. A straightforward computation, using Lemma~\ref{lem:trace-to-F-in-1/2Z}, shows that
	\[
		\phi_1(\pi\overline{\pi}) = \px(a_1,a_2,a_3,a_4) \pm \po(a_1,a_2,a_3,a_4)\sqrt{\Delta_F}.
	\]
	It follows that $\pi$ is a Weil $p$-number for a prime $p$ if and only if the following hold:
	\begin{enumerate}[label=(\roman*)]
		\item $\po(a_1,a_2,a_3,a_4) = 0$
		\item $\px(a_1,a_2,a_3,a_4)$ is prime.
	\end{enumerate}
	
	Next we will show that if (i) holds, then $\ZZ[\pi,\overline{\pi}] = \sO_K$ if and only if the following hold:
	\begin{enumerate}[label=(\roman*)]
		\setcounter{enumi}{2}
		\item $\fl(a_1,a_2,a_3,a_4) = \pm 1$
		\item $\fll(a_1,a_2,a_3,a_4) = \pm 1$.
	\end{enumerate}
	By Lemma~\ref{lem:1-pi-pibar-pi2-basis-for-Z-pi-pibar}, $B = \{1,\pi,\overline{\pi},\pi^2\}$ spans $\ZZ[\pi,\overline{\pi}]$ as a $\ZZ$-module. Therefore $\ZZ[\pi,\overline{\pi}]$ is an order in $K$ if and only if $\disc B \neq 0$, in which case $B$ is basis for $\ZZ[\pi,\overline{\pi}]$. Hence $\ZZ[\pi,\overline{\pi}] = \sO_K$ if and only if $\disc B = \Delta_K$. By Lemma~\ref{lem:disc-1-gamma-gammabar-gamma2} and the definition of $\fl,\fll$,
	\[
		\disc B = \Delta_K\fl(a_1,a_2,a_3,a_4)^4\fll(a_1,a_2,a_3,a_4)^2.
	\]
	By Lemma~\ref{lem:f1,f2,p,p0-coefficients}, $\fl$ and $\fll$ are integer polynomials, so $\ZZ[\pi,\overline{\pi}] = \sO_K$ if and only if $|\fl(a_1,a_2,a_3,a_4)| = |\fll(a_1,a_2,a_3,a_4)| = 1$.

	We have shown that $\pi$ is a super-isolated Weil number if and only if properties (i)-(iv) hold. Because the algorithm enumerates integral tuples $(a_1,a_2,a_3,a_4)$ satisfying these properties, this shows that every algebraic integer the algorithm outputs is a super-isolated Weil number.
	
	For the second claim, suppose that $\pi = \sum a_i\alpha_i$ is a super-isolated Weil number. Then by property (iv), $x_3=a_3$, $x_4=a_4$ is an integral solution to Equation~\ref{cond:f2=pm1}. By Lemma~\ref{lem:stepa3a4-works}, for a large enough bound, the algorithm will eventually enumerate $a_3,a_4$ in step~\ref{step:a3a4}. Recall that $a_1,a_2$ are essentially determined from $a_3,a_4$ (see steps~\ref{step:a2} and \ref{step:a1}). That is, for any given solution $a_3,a_4$ to Equation~\ref{cond:f2=pm1}, there are two pairs $(a_1,a_2)$ such that $(a_1,a_2,a_3,a_4)$ satisfies Equations~(\ref{cond:f2=pm1})-(\ref{cond:p0=0}). Both pairs are found by the algorithm, so the algorithm will eventually output $\pi$.
\end{proof}

\subsubsection{Efficiency}\label{sec:dim-2-alg-runtime}

Recall that the algorithm in Section~\ref{sec:dim-2-alg} enumerates solutions $x_1=a_1$, $x_2=a_2$, $x_3=a_3$, $x_4=a_4$ to Equations~(\ref{cond:f2=pm1})-(\ref{cond:p0=0}). For each integer $i$, chosen in step~\ref{step:choose-i}, the algorithm found several (possibly non-integral) solutions, see steps~(\ref{step:a3a4})-(\ref{step:a1}). In this section, we will show that the algorithm can find all solutions $(a_1,a_2,a_3,a_4)$ with $\px(a_1,a_2,a_3,a_4) \leq N$, by checking at most $O(\log N)$ values of $i$.

To prove the claim, we first show that the value of $|a_4|$ grows exponentially with $|i|$, i.e. $\log|a_4| = \Omega(|i|)$. Then we will show that the function $\px(x_1,x_2,x_3,x_4)$, when restricted to solutions to Equations~(\ref{cond:f2=pm1})-(\ref{cond:p0=0}), is essentially bounded below by $|x_4|$.

\begin{rem}
	The reason we choose $a_4$ instead of $a_3$ is that some of the equations turn out to be simpler. The same argument could be made with $a_3$ instead.
\end{rem}

\begin{lem}\label{lem:bound-a4-from-below-by-i}
	There are computable positive constants $C_1,C_2$, with $C_2 > 1$, such that if the integer $i$ is chosen as in step~\ref{step:choose-i}, and the pair $(a_3,a_4)$, with $a_4 \neq 0$, is computed as in step~\ref{step:compute-a3a4}, then
	\[
		|a_4| \geq C_1 \cdot C_2^{|i|}.
	\]
	The constants $C_1,C_2$ depend only on the basis chosen in step~\ref{step:choose-basis} and the generator chosen in step~\ref{step:choose-generator}.
\end{lem}
\begin{proof}
	We will keep the notation from the algorithm in Section~\ref{sec:dim-2-alg}. Recall how the pair $(a_3,a_4)$ is constructed from $i$ in step~\ref{step:a3a4}. First we found an algebraic integer $\sigma \in \sO_F$ of the form $\sigma = \pm\epsilon^i\gamma$ where $\gamma$ generates the ideal $I = a\ZZ + (b + \sqrt{
	\Delta_F})/2\ZZ$, and $\epsilon$ is a fundamental unit of $F$. The pair $a_3,a_4$ are the coefficients of $\sigma$ with respect to the $\ZZ$-basis $\{a,(b + \sqrt{\Delta_F})/2\}$ for $I$.
	
	Using the quadratic formula and the fact that $a_3,a_4 \in \ZZ$, one can show that $\fll(a_3,a_4) = \pm 1$ implies that $|a_3| \leq C_0|a_4|$ for some constant $C_0$ that depends only on $\fll$ (hence $C_0$ depends on the basis chosen in step~\ref{step:choose-basis}). So
	\begin{align*}
		\left(|a|C_0 + \frac{|b| + \sqrt{\Delta_F}}{2}\right)|a_4|
		&\geq
		|a_3||a| + |a_4|\frac{|b| + \sqrt{\Delta_F}}{2}
		\\
		&\geq
		\max\left(\left|\phi_1(\gamma\epsilon^i)\right|,\left|\phi_2(\gamma\epsilon^i)\right|\right)
		\\
		&\geq
		\min\left(\left|\phi_1\left(\gamma\right)\right|,\left|\phi_2\left(\gamma\right)\right|\right)
		\cdot \max(\left|\phi_1(\epsilon^i)\right|,\left|\phi_2(\epsilon^i)\right|).
		\\
		&=
		\min\left(\left|\phi_1\left(\gamma\right)\right|,\left|\phi_2\left(\gamma\right)\right|\right)
		\cdot
		\max(\left|\phi_1(\epsilon)\right|,\left|\phi_2(\epsilon)\right|)^{|i|}
	\end{align*}
	The last step follows from the fact that $\phi_1(\epsilon)\phi_2(\epsilon) = \pm1$.
\end{proof}

Next we want to show that, for all integral solutions $x_1=a_1$, $x_2=a_2$, $x_3=a_3$, $x_4=a_4$ to Equations~(\ref{cond:f2=pm1})-(\ref{cond:p0=0}), the value of $\px(a_1,a_2,a_3,a_4)$ is essentially bounded below by $|a_4|$. Recall that Equations~\ref{cond:f2=pm1} and \ref{cond:f1=pm1} involve a choice of sign. To simplify our argument, we will first restrict to a specific set of signs.

Let $A$ be the set of rational tuples $(a_1,a_2,a_3,a_4) \in \QQ^4$ such that $x_1=a_1,x_2=a_2,x_3=a_3,x_4=a_4$ is a solution to the following equations:
\begin{align}
	\fl(x_2,x_3,x_4) & = 1
	\label{eq:f1=1}
	\\
	\po(x_1,x_2,x_3,x_4)  &= 0
	\label{eq:p0=+0}
	\\
	x_3 &= \frac{-b + \sqrt{b^2 - 4a(c - 1/x_4^2)}}{2a}x_4.
	\label{eq:x3=x4+sqrt}
\end{align}
Equation~\ref{eq:x3=x4+sqrt} comes from solving $\fll(x_3,x_4) = 1$ for $x_3$ (recall that $\fll = ax_3^2 + bx_3x_4 + cx_4^2$ for integers $a,b,c$). Therefore, every tuple in $A$ is a solution to Equations~(\ref{cond:f2=pm1})-(\ref{cond:p0=0}) with the positive signs. Note that every integral solution to Equations~(\ref{cond:f2=pm1})-(\ref{cond:p0=0}) lies in a set defined in a way similar to $A$, only with a possibly different choice of signs.\footnote{There are a total of $8$ choices of signs we could use to define $A$. These come from the three choices of signs: one in Equation~\ref{cond:f2=pm1}, one in Equation~\ref{cond:f1=pm1}, and one from the quadratic formula when solving Equation~\ref{cond:f2=pm1} for $x_3$. Every solution to Equations~(\ref{cond:f2=pm1})-(\ref{cond:p0=0}) lies in one of these $8$ sets.} Our arguments in the lemmas below will not depend on the choice of sign, so they will apply to any such set.

\begin{lem}\label{lem:find-p4}
	There exists an explicit function $P_4(x_4)$ such that for every $(a_1,a_2,a_3,a_4) \in A$,
	\[
		\px(a_1,a_2,a_3,a_4) = P_4(a_4),
	\]
	where $\px$ is the polynomial from Equation~\ref{eq:p}.
\end{lem}
\begin{proof}
	To prove the claim, we first find functions $g_1(x_4),g_2(x_4),g_3(x_4)$ such that for all tuples $(a_1,a_2,a_3,a_4) \in A$, $a_i = g_i(a_4)$ for $i=1,2,3$. Then we will substitute the $g_i$'s into the polynomial $\px$ in order to construct $P_4(x_4)$.
	
	By definition, if
	\[
		g_3(x_4)
		= \frac{-b + \sqrt{b^2 - 4a(c - 1/x_4^2)}}{2a}x_4,
	\]
	then $a_3 = g_3(a_4)$ for all $(a_1,a_2,a_3,a_4) \in A$.
	
	Next we will find $g_2$. Recall that $\fl(x_2,x_3,x_4)$ is a linear polynomial with a non-zero coefficient of $x_2$ (see step~\ref{step:a2} in the algorithm in Section~\ref{sec:dim-2-alg}). So we can use Equation~\ref{eq:f1=1} to write $x_2$ as a linear function of $x_3$ and $x_4$. By substituting $g_3$ for $x_3$, we obtain a function $g_2$ which satisfies $a_2 = g_2(a_4)$ for all $(a_1,a_2,a_3,a_3) \in A$.
	
	Now we will find $g_1$. Recall that $\po(x_1,x_2,x_3,x_4) = \fl x_1 + g$ for some $g\in\ZZ[x_2,x_3,x_4]$ (see step~\ref{step:a1}). By Equations~\ref{eq:f1=1} and \ref{eq:p0=+0}, $a_1 = -g(a_2,a_3,a_4)$ for all $(a_1,a_2,a_3,a_4) \in A$. By replacing $x_2,x_3$ in $-g$ with $g_2,g_3$ respectively, we obtain a function $g_1(x_4)$ such that $a_1 = g_1(a_4)$ for all $(a_1,a_2,a_3,a_4) \in A$.
	
	Let
	\[
		P_4(x_4) = \px(g_1(x_4), g_2(x_4), g_3(x_4), x_4).
	\]
	Note that $P_4$ has the desired property because for all $(a_1,a_2,a_3,a_4) \in A$, we have that $g_i(a_4) = a_i$ for $i=1,2,3$.
\end{proof}

\begin{thm}\label{thm:algorithm-bound}
	The algorithm in Section~\ref{sec:dim-2-alg} can find all super-isolated Weil $p$-numbers with $p \leq N$ in $O(\log N)$ steps. That is, for each quartic CM field $K$ of class number $1$, there is at least one set of choices that can be made in steps~\ref{step:choose-basis} and \ref{step:choose-embeddings} such that the algorithm only needs to check $O(\log N)$ values of $|i|$ in step~\ref{step:choose-i}. Here the implicit constant depends on the choices made.
\end{thm}
\begin{proof}[Sketch of proof]
	By Theorem~\ref{thm:g=2-algorithm-works}, the algorithm will eventually output any specific super-isolated Weil number given a large enough bound. Therefore, it is sufficient to show that the value of $\px(a_1,a_2,a_3,a_4)$ in step~\ref{step:check-output} grows exponentially in $|i|$. From Lemma~\ref{lem:bound-a4-from-below-by-i}, we know that $\log |a_4| = \Omega(|i|)$, so it is sufficient to show that $\px(a_1,a_2,a_3,a_4) = \Omega(|a_4|)$.
	
	Recall that there are only $91$ such fields. For each one, we computed the function $P_4$ from Lemma~\ref{lem:find-p4} after choosing some random values in steps~\ref{step:choose-basis} and \ref{step:choose-embeddings}. We found that $|P_4(x_4)| = \Omega(x_4^4)$. We repeated the calculations for every alternative definition of the set $A$ from Lemma~\ref{lem:find-p4}, and found the same result (this property seems to always hold in practice). Some details for the case of $K = \QQ(\zeta_5)$ are given in Appendix~\ref{sec:details-for-zeta5}. We also used these calculations in the proof of Theorem~\ref{thm:only-2-examples} below.
	
	Let $(a_1,a_2,a_3,a_4)$ be any integral solution to Equations~(\ref{cond:f2=pm1})-(\ref{cond:p0=0}). Then $\px(a_1,a_2,a_3,a_4) = P_4(a_4)$ for some $P_4$ (recall the definition of $P_4$ depended on the set $A$, so there are $8$ possibilities for $P_4$). Since $|P_4(x_4)| = \Omega(x_4^4)$, it follows that $\px(a_1,a_2,a_3,a_4) = \Omega(a_4^4)$.
\end{proof}

\subsection{Examples}\label{sec:dim-2-examples}

We found the following super-isolated Weil numbers by using the algorithm in Section~\ref{sec:dim-2-alg}. Both generate non-normal quartic CM fields.

\begin{ex}\label{ex:iso-surface}
	{\tiny\[
		\pi = \frac{225058681}{16}\left(\sqrt{-19 - 8\sqrt{2}}\right)^3 + \frac{1}{16}\left(-19 - 8\sqrt{2}\right) + \frac{6822363251}{16}\sqrt{-19 - 8\sqrt{2}} - \frac{4404669978983883573}{16}.
	\]}
	Here $p = \pi\overline{\pi} = 75785615717819865717549739169971883$ is a 116 bit prime, and $N=\norm_{K/\QQ}(\pi-1)$ factors as $31$ times a $227$ bit prime. The associated surface is the jacobian of the following hyperelliptic curve over $\FF_p$:
	{\tiny\begin{align*}
		y^{2} &= 518974905053625554694780 x^{6} + 1102935355117356837110620 x^{5} + 991287292238024940555812 x^{4}
		\\ &\quad + 478588249786621434333076 x^{3} + 130273203505281201694544 x^{2} + 19179534443912344652288 x
		\\ &\quad + 1373526256863485541624.
		\end{align*}}
\end{ex}

\begin{ex}\label{ex:iso-surface2}
	{\tiny\[
		\pi = \frac{701408733}{8}\left(\sqrt{-13 - 2\sqrt{5}}\right)^3 - \frac{1}{8}\left(-13 - 2\sqrt{5}\right) + \frac{12255108743}{8}\sqrt{-13 - 2\sqrt{5}} + \frac{18762798022945344405}{8}.
	\]}
	Here $p = \pi\overline{\pi} = 5500665463278776959453617590160336793$ is a 123 bit prime, and $N = \norm_{K/\QQ}(\pi - 1)$ factors as $521$ times a $236$ bit prime. The associated surface is the jacobian of the following hyperelliptic curve over $\FF_p$:
	{\tiny\begin{align*}
		y^{2} &=
		3166541774481651094230166870474839614x^6
		+ 153452867072273239090020172039655416x^5
		\\ &\quad
		+ 4397111106428325553768487123769953829x^4
		+ 4136411707045872026156847617680586720x^3
		\\ &\quad
		+ 801646319360879802078118801683649366x^2
		+ 3958303885280886436811484306434693399x
		\\ &\quad
		+ 2303639253886822235537433002764323459.
		\end{align*}}
\end{ex}

\subsection{Main result}\label{sec:dim-2-only-examples}

Our main result is that Examples~\ref{ex:iso-surface} and \ref{ex:iso-surface2} are the only examples of super-isolated surfaces with near-prime order and cryptographic size.

\begin{thm}\label{thm:only-2-examples}
	Examples~\ref{ex:iso-surface} and \ref{ex:iso-surface2} are the only super-isolated abelian surfaces $A/\FF_p$ with the property that
	\begin{equation}\label{cond:near-prime-and-crypto-size}
		\#A(\FF_p) = cr \text{ where } c \leq 1000, \ r \text{ is prime}, \text{ and } 2^{160} \leq r \leq 2^{512}.
	\end{equation}
\end{thm}
\begin{proof}[Sketch of proof]
	We will show that if $A$ is an abelian surface satisfying property~(\ref{cond:near-prime-and-crypto-size}), then the roots of the characteristic polynomial of the Frobenius endomorphism of $A$ are super-isolated Weil $p$-numbers with $p \leq 2^{261}$. The claim then follows by running the algorithm in Section~\ref{sec:dim-2-alg} long enough to find all super-isolated Weil $p$-numbers with $p \leq 2^{261}$.
	
	Let $A$ be an abelian surface over $\FF_p$ satisfying property~(\ref{cond:near-prime-and-crypto-size}). Recall from Section~\ref{sec:dim-2-background} that $\#A(\FF_p) = f(1)$, where $f$ is the characteristic polynomial of the Frobenius endomorphism of $A$. Using the well-known Hasse bound, it is not hard to show that property~(\ref{cond:near-prime-and-crypto-size}) implies that $A$ is simple, because a product of elliptic curves does not have near-prime order. Hence $f = h^e$ for some irreducible polynomial $h$, see Section~\ref{sec:dim-2-background}. But as $f(1)$ is near-prime, we must have $e=1$, i.e. $f$ is irreducible. So by Corollary~\ref{cor:surfaces-super-isolated-condition}, every root $\pi$ of $f$ is a super-isolated Weil $p$-number in the quartic CM field $K = \QQ(\pi)$ of class number $1$.
	
	Next we will show that property~\ref{cond:near-prime-and-crypto-size} implies that $p \leq 2^{261}$. This is similar to using the Hasse bound above. Since the roots of $f$ are Weil $p$-numbers, it follows that $f(x) = x^4 + ax^3 + bx^2 + pax + p^2$ with $|a| \leq 4\sqrt{p}$ and $|b| \leq 6p$. So
	\[
		r \geq \frac{p^2 - 4p^{3/2} - 6p - 4\sqrt{p} - 1}{1000}.
	\]
	A straightforward calculation shows that this inequality, when combined with the bound $r \geq 2^{512}$, implies that $p \leq 2^{261}$.
	
	The next part of the proof is computational. We used {\tt Sage} to compute the implicit constants in Section~\ref{sec:dim-2-alg-runtime} that are used to bound the number of steps the algorithm must take in order to enumerate all super-isolated Weil $p$-numbers with $p \leq 2^{261}$ (see Theorem~\ref{thm:algorithm-bound}). Some details for the case of $K = \QQ(\zeta_5)$ are given in Appendix~\ref{sec:details-for-zeta5}. Our results show that there are $282$ conjugacy classes of such Weil numbers. Only those given in Examples~\ref{ex:iso-surface} and \ref{ex:iso-surface2} satisfy the properties in the claim. The source code is available at \url{https://sites.math.washington.edu/~tscholl2/super-isolated}.
\end{proof}

\begin{rem}
	The bound $c \leq 1000$ used above is arbitrary. The smaller $c$ is the more efficient the cryptosystem will be. The three smallest values of $c$ were $31$, $521$, and $73399$.
\end{rem}

\begin{rem}
	Note that the surfaces in Examples~\ref{ex:iso-surface} and \ref{ex:iso-surface2} provide $113$ and $116$ bits of security respectively (i.e. half the bitlength of the largest prime dividing the order). Recent standards suggest using between $128$ and $256$ bits of security \cite{fips2013}. While $113$ bits should be fine in many cases, our result shows that we can increase the security without increasing the size of $p$ or $c$, which would reduce the efficiency.
\end{rem}

\begin{rem}
	After running the algorithm to search all super-isolated Weil $p$-numbers $\pi$ with $100 \leq p \leq 2^{10000}$, there was only one conjugacy class such that $\norm(\pi - 1)$ was prime. In that case, $p \approx 2^{740}$. This prime is too large to be useful in most practical applications.
\end{rem}

\appendix

\section{Details for $\QQ(\zeta_5)$}\label{sec:details-for-zeta5}

In this section, we provide a detailed example of the algorithm in Section~\ref{sec:dim-2-alg}, for the field $\QQ(\zeta_5)$. We also show how long the algorithm must run in order to enumerate all super-isolated Weil $p$-numbers with $p \leq 2^{261}$. We will use the notation from Section~\ref{sec:dim-2-alg} and the methods from Section~\ref{sec:dim-2-alg-runtime}.

Let $K = \QQ(\zeta_5)$.
Recall that step~\ref{step:choose-basis} of the algorithm chooses a basis $\{\alpha_1,\alpha_2,\alpha_3,\alpha_4\}$ for $\sO_K$ such that $\{\alpha_1,\alpha_2\}$ form a basis for $\sO_F$ where $F$ is the maximal real subfield of $K$. In this case, $F = \QQ(\sqrt{5})$. We choose
\[
	\alpha_1 = 1, \quad
	\alpha_2 = -\zeta_5^3 - \zeta_5^2 + 2, \quad
	\alpha_3 = -3\zeta_5^2 - 2\zeta_5^2 - 2, \quad
	\alpha_4 = -2\zeta_5^3 + 3\zeta_5^2 - \zeta_5 - 1.
\]
Let $\phi_1,\phi_2$ be the embeddings $K \hookrightarrow \CC$ defined by
\[
	\phi_1(\zeta_5) = e^{2\pi i/5} ,\quad
	\phi_2(\zeta_5) = e^{4\pi i/5}.
\]
For reference, $\phi_1(\alpha_2) = \frac{5 + \sqrt{5}}{2}$ and $\phi_2(\alpha_2) = \frac{5 - \sqrt{5}}{2}$.

Next we compute the polynomials defined in Equations~(\ref{eq:f1})-(\ref{eq:p0}):
\begin{align*}
	\fl &=
	2 x_{2} + 5 x_{3} - 2 x_{4}
	\\
	\fll &=
	x_{3}^{2} + 9 x_{3} x_{4} + 19 x_{4}^{2}
	\\
	\px &=
	x_{1}^{2} + 5 x_{1} x_{2} + \frac{15}{2} x_{2}^{2} - \frac{3}{2} x_{1} x_{3} + \frac{5}{2} x_{2} x_{3} + 9 x_{3}^{2} - 2 x_{1} x_{4} - \frac{15}{2} x_{2} x_{4} + \frac{3}{2} x_{3} x_{4} + \frac{37}{2} x_{4}^{2}
	\\
	\po &=
	x_{1} x_{2} + \frac{5}{2} x_{2}^{2} + \frac{5}{2} x_{1} x_{3} + \frac{11}{2} x_{2} x_{3} - 2 x_{3}^{2} -  x_{1} x_{4} - \frac{7}{2} x_{2} x_{4} - \frac{7}{2} x_{3} x_{4} - \frac{9}{2} x_{4}^{2}.
\end{align*}

The next step is to enumerate all solutions to the equation $\fll(x_3,x_4) = \pm 1$. We do this following the method laid out in step~\ref{step:a3a4}.

Write $\fll = ax_3^2 + bx_3x_4 + cx_4^2$, and then choose a generator $\gamma$ for the ideal $I = a\ZZ + ((b + \sqrt{\Delta_F})/2)\ZZ$. Because in this case $I = \sO_F$, we will choose $\gamma = 1$.
A fundamental unit for $\sO_F$ is
\[
	\epsilon = -\zeta_5^3 - \zeta_5^2 - 1.
\]

Now for each choice of sign and value of $i$ we compute $\sigma = \pm \epsilon^i\gamma$ and write
\[
	\sigma = a_3a + a_4\frac{b+\sqrt{\Delta_F}}{2}
\]
as in step~\ref{step:compute-a3a4}.

The remaining steps of the algorithm are straightforward, so we will skip ahead to compute the bound on $i$ in order to find all super-isolated Weil $p$-numbers with $p \leq 2^{261}$. We first compute the constants from Section~\ref{sec:dim-2-alg-runtime}.

Notice that $\fll(a_3,a_4) = \pm 1$ implies that
\begin{align*}
	a_3
	&= \frac{-ba_4 \pm \sqrt{(ba_4)^2 - 4a(ca_4^2 \pm 1)}}{2a}
	\\
	&= \frac{-ba_4 \pm |a_4|\sqrt{\Delta_F \pm 4a/a_4^2}}{2a}.
\end{align*}
Therefore we can bound $|a_3|$ from above by
\begin{align*}
	|a_3|
	&\leq \frac{|b| + \sqrt{\Delta_F + 4|a|}}{2|a|} |a_4|
	\\
	&= 6|a_4|
\end{align*}

Next we want to bound $|a_4|$ from below by an exponential function in $i$. Following the proof of Lemma~\ref{lem:bound-a4-from-below-by-i},
\begin{align*}
	\min(|\phi_1(\gamma)|,|\phi_2(\gamma|)\max(|\phi_1(\epsilon)|,|\phi_2(\epsilon)|)^{|i|}
	&\leq \max\left(|\phi_1(\gamma\epsilon^i)|,|\phi_2(\gamma\epsilon^i)|\right)
	\\
	&\leq |a_3||a| + |a_4|\frac{|b| + \sqrt{\Delta_F}}{2}
	\\
	&= |a_3| + \frac{9 + \sqrt{5}}{2}|a_4|
	\\
	&\leq 6|a_4| + \frac{9 + \sqrt{5}}{2}|a_4|
	\\
	&\leq 12|a_4|
\end{align*}
Since $\gamma = 1$ and $\max(|\phi_1(\epsilon)|,|\phi_2(\epsilon)|) = \frac{1 + \sqrt{5}}{2} \geq 3/2$, we have
\[
	|a_4| \geq \frac{1}{12}\left(\frac{3}{2}\right)^{|i|}.
\]

Next we want to bound $\px(a_1,a_2,a_3,a_4)$ from below in terms of $|a_4|$. We follow the same steps as in the proof of Lemma~\ref{lem:find-p4}. For simplicity, we will assume that $(a_1,a_2,a_3,a_4)$ is a solution to Equations~(\ref{eq:f1=1})-(\ref{eq:x3=x4+sqrt}). In our case, these equations are
\begin{align*}
	1 &= 2x_2 + 5x_3 - 2x_4
	\\
	0 &= x_1 + 5x_2^2 + 11x_2x_3 - 4x_3^2 - 7x_2x_4 - 7x_3x_4 - 9x_4^2
	\\
	x_3 &= \frac{-9 + \sqrt{5 + 4/x_4^2}}{2}x_4
\end{align*}
Note that integral solutions to these equations exist, for example, $x_1=115,x_2=-45,x_3=17,x_4=-3$. This solution comes from setting $i = 7$.

By solving each of these constraints in terms of $x_4$, we have
\begin{align*}
	x_1 &=
	\frac{5}{8} \, x_4^{2} + \frac{1}{8} \, {\left(5 \, x_4^{2} + 28 \, x_4\right)} \sqrt{5 + \frac{4}{x_4^2}} - 33 \, x_4 - 1
	\\
	x_2 &=
	-\frac{5}{4} \, x_4 \sqrt{5 + \frac{4}{x_4^2}} + \frac{49}{4} \, x_4 + \frac{1}{2}
	\\
	x_3 &=
	\frac{1}{2} \, x_4 \sqrt{5 + \frac{4}{x_4^2}} - \frac{9}{2} \, x_4
\end{align*}

Substituting these into the polynomial $\px(x_1,x_2,x_3,x_4)$ as in the proof of Lemma~\ref{lem:find-p4}, we have
\[
	P_4(x_4) =
	\frac{75}{32} \, x_4^{4} + \frac{55}{16} \, x_4^{2} + \frac{5}{32} \, {\left(5 \, x_4^{4} + 4 \, x_4^{2}\right)} \sqrt{5 + \frac{4}{x_4^2}} + 1
\]
In particular,
\[
	P_4(x_4) \geq \frac{75}{32}x_4^4
\]
for all $x_4 > 0$.

We saw above that for each integer $i$, any associated pair $(a_3,a_4)$ (as computed in step~\ref{step:a3a4} of the algorithm in Section~\ref{sec:dim-2-alg}) satisfies
\[
	|a_4| \geq \frac{1}{12}\left(\frac{3}{2}\right)^{|i|}.
\]
If the tuple happens to satisfy Equations~(\ref{eq:f1=1})-(\ref{eq:x3=x4+sqrt}), then
\[
	\px(a_1,a_2,a_3,a_4) \geq \frac{75}{32}\left(\frac{1}{12}\left(\frac{3}{2}\right)^{|i|}\right)^4.
\]
So in order to capture all super-isolated Weil $p$-numbers with $p \leq 2^{261}$ which satisfy the constraints above, we have to check all $i$ with $|i| \leq 118$.
The bound for other choices of signs in Equations~(\ref{eq:f1=1})-(\ref{eq:x3=x4+sqrt}) can be computed similarly.

\bibliographystyle{plain}
\bibliography{references}

\begin{thebibliography}{10}

\bibitem{bateman1962heuristic}
Paul~T. Bateman and Roger~A. Horn.
\newblock A heuristic asymptotic formula concerning the distribution of prime
  numbers.
\newblock {\em Math. Comp.}, 16:363--367, 1962.

\bibitem{bernstein2014kummer}
Daniel~J. Bernstein, Chitchanok Chuengsatiansup, Tanja Lange, and Peter
  Schwabe.
\newblock Kummer strikes back: new {DH} speed records.
\newblock In {\em Advances in cryptology---{ASIACRYPT} 2014. {P}art {I}},
  volume 8873 of {\em Lecture Notes in Comput. Sci.}, pages 317--337. Springer,
  Heidelberg, 2014.

\bibitem{cahen1997integer}
Paul-Jean Cahen and Jean-Luc Chabert.
\newblock {\em Integer-valued polynomials}, volume~48 of {\em Mathematical
  Surveys and Monographs}.
\newblock American Mathematical Society, Providence, RI, 1997.

\bibitem{cohen1993course}
Henri Cohen.
\newblock {\em A course in computational algebraic number theory}, volume 138
  of {\em Graduate Texts in Mathematics}.
\newblock Springer-Verlag, Berlin, 1993.

\bibitem{cohen2000advanced}
Henri Cohen.
\newblock {\em Advanced topics in computational number theory}, volume 193 of
  {\em Graduate Texts in Mathematics}.
\newblock Springer-Verlag, New York, 2000.

\bibitem{cohen2006handbook}
Henri Cohen, Gerhard Frey, Roberto Avanzi, Christophe Doche, Tanja Lange, Kim
  Nguyen, and Frederik Vercauteren, editors.
\newblock {\em Handbook of elliptic and hyperelliptic curve cryptography}.
\newblock Discrete Mathematics and its Applications (Boca Raton). Chapman \&
  Hall/CRC, Boca Raton, FL, 2006.

\bibitem{enge2002computing}
Andreas Enge.
\newblock Computing discrete logarithms in high-genus hyperelliptic {J}acobians
  in provably subexponential time.
\newblock {\em Math. Comp.}, 71(238):729--742, 2002.

\bibitem{gaudry2007double}
P.~Gaudry, E.~Thom\'e, N.~Th\'eriault, and C.~Diem.
\newblock A double large prime variation for small genus hyperelliptic index
  calculus.
\newblock {\em Math. Comp.}, 76(257):475--492, 2007.

\bibitem{gaudry2000algorithm}
Pierrick Gaudry.
\newblock An algorithm for solving the discrete log problem on hyperelliptic
  curves.
\newblock In {\em Advances in cryptology---{EUROCRYPT} 2000 ({B}ruges)}, volume
  1807 of {\em LNCS}, pages 19--34. Springer, Berlin, 2000.

\bibitem{honda1968isogeny}
Taira Honda.
\newblock Isogeny classes of abelian varieties over finite fields.
\newblock {\em J. Math. Soc. Japan}, 20:83--95, 1968.

\bibitem{koblitz2011elliptic}
Ann~Hibner Koblitz, Neal Koblitz, and Alfred Menezes.
\newblock Elliptic curve cryptography: the serpentine course of a paradigm
  shift.
\newblock {\em J. Number Theory}, 131(5):781--814, 2011.

\bibitem{koblitz1989hyperelliptic}
Neal Koblitz.
\newblock Hyperelliptic cryptosystems.
\newblock {\em J. Cryptology}, 1(3):139--150, 1989.

\bibitem{louboutin1994determination}
Stéphane Louboutin and Ryotaro Okazaki.
\newblock Determination of all non-normal quartic {CM}-fields and of all
  non-abelian normal octic {CM}-fields with class number one.
\newblock {\em Acta Arith.}, 67(1):47--62, 1994.

\bibitem{maisner2002abelian}
Daniel Maisner and Enric Nart.
\newblock Abelian surfaces over finite fields as {J}acobians.
\newblock {\em Experiment. Math.}, 11(3):321--337, 2002.
\newblock With an appendix by Everett W. Howe.

\bibitem{menezes2006cryptographic}
Alfred Menezes and Edlyn Teske.
\newblock Cryptographic implications of {H}ess' generalized {GHS} attack.
\newblock {\em Appl. Algebra Engrg. Comm. Comput.}, 16(6):439--460, 2006.

\bibitem{menezes1993reducing}
Alfred~J. Menezes, Tatsuaki Okamoto, and Scott~A. Vanstone.
\newblock Reducing elliptic curve logarithms to logarithms in a finite field.
\newblock {\em IEEE Trans. Inform. Theory}, 39(5):1639--1646, 1993.

\bibitem{menezes1997handbook}
Alfred~J. Menezes, Paul~C. van Oorschot, and Scott~A. Vanstone.
\newblock {\em Handbook of applied cryptography}.
\newblock CRC Press Series on Discrete Mathematics and its Applications. CRC
  Press, Boca Raton, FL, 1997.
\newblock With a foreword by Ronald L. Rivest.

\bibitem{miele2015efficient}
Andrea Miele and Arjen~K Lenstra.
\newblock Efficient ephemeral elliptic curve cryptographic keys.
\newblock In {\em Information Security: 18th International Conference, ISC
  2015, Trondheim, Norway, September 9-11, 2015, Proceedings}, volume 9290 of
  {\em LNCS}, pages 524--547. Springer International Publishing, 2015.

\bibitem{fips2013}
{National Institute of Standards and Technology}.
\newblock Digital {S}ignature {S}tandard ({DSS}).
\newblock Federal Information Processing Standards Publication 186-4, 2013.

\bibitem{neukrich1999algebraic}
J{\"u}rgen Neukirch.
\newblock {\em Algebraic number theory}, volume 322 of {\em Grundlehren der
  Mathematischen Wissenschaften [Fundamental Principles of Mathematical
  Sciences]}.
\newblock Springer-Verlag, Berlin, 1999.
\newblock Translated from the 1992 German original and with a note by Norbert
  Schappacher, With a foreword by G. Harder.

\bibitem{schoof1987nonsingular}
Ren{\'e} Schoof.
\newblock Nonsingular plane cubic curves over finite fields.
\newblock {\em J. Combin. Theory Ser. A}, 46(2):183--211, 1987.

\bibitem{silverman2009arithmetic}
Joseph~H. Silverman.
\newblock {\em The arithmetic of elliptic curves}, volume 106 of {\em Graduate
  Texts in Mathematics}.
\newblock Springer, Dordrecht, second edition, 2009.

\bibitem{smith2008isogenies}
Benjamin Smith.
\newblock Isogenies and the discrete logarithm problem in {J}acobians of genus
  3 hyperelliptic curves.
\newblock volume~22, pages 505--529. 2009.

\bibitem{stark1967complete}
H.~M. Stark.
\newblock A complete determination of the complex quadratic fields of
  class-number one.
\newblock {\em Michigan Math. J.}, 14:1--27, 1967.

\bibitem{effective1974stark}
H.~M. Stark.
\newblock Some effective cases of the {B}rauer-{S}iegel theorem.
\newblock {\em Invent. Math.}, 23:135--152, 1974.

\bibitem{sage}
{The Sage Developers}.
\newblock {\em SageMath, the Sage Mathematics Software System (Version 7.5)},
  2017.
\newblock {\tt http://www.sagemath.org}.

\bibitem{wenhan2012isolated}
Wenhan Wang.
\newblock {\em Isolated {C}urves for {H}yperelliptic {C}urve {C}ryptography}.
\newblock PhD thesis, University of Washington, 2012.

\bibitem{waterhouse1971abelian}
W.~C. Waterhouse and J.~S. Milne.
\newblock Abelian varieties over finite fields.
\newblock In {\em 1969 {N}umber {T}heory {I}nstitute ({P}roc. {S}ympos. {P}ure
  {M}ath., {V}ol. {XX}, {S}tate {U}niv. {N}ew {Y}ork, {S}tony {B}rook,
  {N}.{Y}., 1969)}, pages 53--64. Amer. Math. Soc., Providence, R.I., 1971.

\bibitem{waterhouse1969abelian}
William~C. Waterhouse.
\newblock Abelian varieties over finite fields.
\newblock volume~2, pages 521--560, 1969.

\bibitem{yamamura1994determination}
Ken Yamamura.
\newblock The determination of the imaginary abelian number fields with class
  number one.
\newblock {\em Math. Comp.}, 62(206):899--921, 1994.

\end{thebibliography}
	
\end{document}